\documentclass[12pt,reqno]{amsart}
\usepackage{amsmath,color}	
\usepackage[hypertex,bookmarks=true]{hyperref}
\def\UU{{\mathcal S}}
\def\DD{{\mathcal D}}
\def\bP{{\bf P}}
\def\R{{\mathbb R}}
\def\N{{\mathbb N}}
\def\zero{{\bf 0}}
\def\ee{{\bf e}}
\def\xx{{\bf x}}
\def\yy{{\bf y}}
\def\zz{{\bf z}}
\def\aa{{\bf a}}
\def\bb{{\bf b}}
\def\cc{{\bf c}}

\def\mm{{\bf m}}

\def\one{{\bf 1}}
\def\supp{\mathop{\rm supp}}
\def\myor{\mathop{\vee}}

\newtheorem*{prob}{Problem}
\newtheorem*{conj}{Conjecture}
\newtheorem{claim}{Claim}
\newtheorem{la}{Lemma}
\newtheorem{thm}{Theorem}
\newtheorem{cor}[thm]{Corollary}

\title{Multiply union families in $\N^n$}
\author[P.~Frankl]{
Peter Frankl}
\address{
Alfr\'ed R\'enyi Institute of Mathematics,
H-1364 Budapest, P.O.Box 127, Hungary
}
\email{peter.frankl@gmail.com}
\author[M.~Shinohara]{Masashi Shinohara}
\address{Faculty of Education, Shiga University, 2-5-1 Hiratsu, Shiga 520-0862, Japan}
\email{shino@edu.shiga-u.ac.jp}
\author[N.~Tokushige]{Norihide Tokushige}
\thanks{The third author was supported by JSPS KAKENHI 25287031.}
\address{
College of Education, Ryukyu University\\
Nishihara, Okinawa 903-0213, Japan
}
\email{hide@edu.u-ryukyu.ac.jp}
\date{\today}

\begin{document}
\begin{abstract}
Let $A\subset \N^{n}$ be an $r$-wise $s$-union family, that is, 
a family of sequences with $n$ components of non-negative integers
such that for any $r$ sequences in $A$ the total sum of the maximum of 
each component in those sequences is at most $s$.
We determine the maximum size of $A$ and its unique extremal configuration
provided (i) $n$ is sufficiently large for fixed $r$ and $s$, or
(ii) $n=r+1$.
\end{abstract}

\maketitle


\section{Introduction}
Let $\N:=\{0,1,2,\ldots\}$ denote the set of non-negative integers,
and let $[n]:=\{1,2,\ldots,n\}$.
Intersecting families in $2^{[n]}$ or $\{0,1\}^n$ are one of the main objects in
extremal set theory. The equivalent dual form of an intersecting family
is a union family, which is the subject of this paper.
In \cite{FT} Frankl and Tokushige proposed to consider such problems not 
only in $\{0,1\}^n$ but also in $[q]^n$. They determined
the maximum size of 2-wise $s$-union families (i) in $[q]^n$ for $n>n_0(q,s)$,
and (ii) in $\N^3$ for all $s$ (the definitions will be given shortly). 
In this paper we extend their results
and determine the maximum size and structure of $r$-wise $s$-union 
families in $\N^{n}$ for the following two cases: (i) $n\geq n_0(r,s)$,
and (ii) $n=r+1$.
Much research has been done for the case of families in $\{0,1\}^n$,
and there are many challenging open problems.
The interested reader is referred to \cite{F1,F2,FT0,T1,T2}.

For a vector $\xx\in\R^n$, we write $x_i$ or $(\xx)_i$
for the $i$th component, so $\xx=(x_1,x_2,\ldots,x_n)$.
Define the {\sl weight} of $\aa\in\N^n$ by
\[
 |\aa|:=\sum_{i=1}^n a_i.
\]
For a finite number of vectors $\aa,\bb,\ldots,\zz\in \N^n$ define the 
join $\aa\myor \bb\myor\cdots\myor\zz$ by
\[
 (\aa\myor \bb\myor\cdots\myor\zz)_i:=\max\{a_i,b_i,\ldots,z_i\},
\]
and we say that $A\subset \N^n$ is $r$-wise $s$-{\sl union} if
\[
 |\aa_1\myor \aa_2\myor\cdots\myor\aa_r|\leq s\text{ for all }
\aa_1,\aa_2,\ldots,\aa_r\in A.
\]
In this paper we address the following problem.

\begin{prob}
For given $n,r$ and $s$, determine the maximum size $|A|$ of $r$-wise
$s$-union families $A\subset\N^n$.
\end{prob}

To describe candidates $A$ that give the maximum size to the above problem,
we need some more definitions.
Let us introduce a partial order $\prec$ in $\R^n$. 
For $\aa, \bb\in \R^n$ we let $\aa\prec \bb$ iff
$a_i\leq b_i$ for all $1\leq i\leq n$.
Then we define a {\sl down set} for $\aa\in \N^n$ by
\[
 \DD(\aa):=\{\cc\in \N^n: \cc\prec \aa\},
\]
and for $A\subset \N^n$ let 
\[
 \DD(A):=\bigcup_{\aa\in A}\DD(\aa).
\]
We also introduce $\UU(\aa,d)$, which can be viewed as a part of sphere 
centered at $\aa\in\N^n$ with radius $d\in\N$, defined by
\[
\UU(\aa,d):=\{\aa+{\boldsymbol \epsilon}\in \N^n: 
{\boldsymbol\epsilon}\in \N^n,\, |{\boldsymbol\epsilon}|=d\}.
\]
We say that $\aa\in \N^n$ is a {\sl balanced partition}, 
if all $a_i$'s are as close to each other as possible, more precisely,
$|a_i-a_j|\leq 1$ for all $i,j$.
Let $\one:=(1,1,\ldots,1)\in \N^n$. 

For $r,s,n,d\in\N$ with $0\leq d\leq\lfloor\frac sr\rfloor$ 
and $\aa\in\N^n$ with $|\aa|=s-rd$ let us define a family $K$ by
\begin{equation}\label{def:K}
K= K(r,n,\aa,d):=
\bigcup_{i=0}^{\lfloor\frac du\rfloor}\DD(\UU(\aa+i\one,d-ui)),
\end{equation}
where $u=n-r+1$. 
This is the candidate family. Intuitively $K$ is a union of balls, 
and the corresponding centers and radii are chosen so that $K$ is 
$r$-wise $s$-union as we will see in Claim~\ref{claim1} in the next section.

\begin{conj}
Let $r\geq 2$ and $s$ be positive integers.
If $A\subset\N^{n}$ is $r$-wise $s$-union, then 
\[
 |A|\leq\max_{0\leq d\leq\lfloor\frac sr\rfloor}
\left| K(r,n,\aa,d)\right|,
\]
where $\aa\in\N^{n}$ is a balanced partition with $|\aa|=s-rd$.
Moreover if equality holds, then $A=K(r,n,\aa,d)$ for some 
$0\leq d\leq\lfloor\frac sr\rfloor$.
\end{conj}

We first verify the conjecture when $n$ is sufficiently large for fixed $r,s$. 
Let $\ee_i$ be the $i$-th standard base of $\R^n$, that is, 
$(\ee_i)_j=\delta_{ij}$. Let $\tilde\ee_0=\zero$, and
$\tilde\ee_i=\sum_{j=1}^i\ee_j$ for $1\leq i\leq n$, e.g., $\tilde\ee_n=\one$.

\begin{thm}\label{thm1}
Let $r\geq 2$ and $s$ be fixed positive integers. 
Write $s=dr+p$ where $d$ and $p$ are non-negative integers with $0\leq p<r$.
Then there exists an $n_0(r,s)$ such that if $n>n_0(r,s)$ 
and $A\subset \N^n$ is $r$-wise $s$-union, then 
\[
 |A|\leq\left| \DD(\UU(\tilde\ee_p,d))\right|.
\]
Moreover if equality holds, then $A$ is isomorphic to 
$\DD(\UU(\tilde\ee_p,d))=K(r,n,\tilde\ee_p,d)$.
\end{thm}

We mention that the case $A\subset\{0,1\}^n$ of 
Conjecture is posed in \cite{F1} and partially solved in \cite{F1,F2},
and the case $r=2$ of Theorem~\ref{thm1} is proved in \cite{FT}
in a slightly stronger form.
We also notice that if $A\subset\{0,1\}^n$ is
2-wise $(2d+p)$-union, then the Katona's $t$-intersection theorem \cite{Katona}
states that $|A|\leq |\DD(\UU(\tilde\ee_p,d)\cap\{0,1\}^n)|$ for all $n\geq s$.

Next we show that the conjecture is true if $n=r+1$. 
We also verify the conjecture on general $n$ if $A$ satisfies some additional 
properties described below.

Let $A\subset\N^{n}$ be $r$-wise $s$-union. For $1\leq i\leq n$ let
\begin{equation}\label{def:mi}
 m_i:=\max\{x_i:\xx\in A\}. 
\end{equation}
If $n-r$ divides $|\mm|-s$, then we define
\begin{equation}\label{def:d}
 d:=\frac{|\mm|-s}{n-r}\geq 0,  
\end{equation}
and for $1\leq i\leq n$ let
\begin{equation}\label{def:ai}
a_i:=m_i-d, 
\end{equation}
and we assume that $a_i\geq 0$.
In this case we have $|\aa|=s-rd$. 
Since $|\aa|\geq 0$ it follows that
$d\leq\lfloor\frac sr\rfloor$.
For $1\leq i\leq n$ define $P_i\in\N^n$ by
\begin{equation}\label{def:Pi}
 P_i:=\aa+d\ee_i, 
\end{equation}
where $\ee_i$ denotes the $i$th standard base, for example,
$P_2=(a_1,a_2+d,a_3,\ldots,a_n)$.

\begin{thm}\label{mainthm}
Let $A\subset\N^{n}$ be $r$-wise $s$-union.
Assume that the sequences $P_i$ are well-defined and
\begin{equation}\label{assumption}
\{P_1,\ldots,P_n\}\subset A. 
\end{equation}
Then it follows that
\[
 |A|\leq\max_{0\leq d'\leq\lfloor\frac sr\rfloor}
\left| K(r,n,\aa',d')\right|,
\]
where $\aa'\in\N^{n}$ is a balanced partition with $|\aa'|=s-rd'$.
Moreover if equality holds, then
$A=K(r,n,\aa',d')$ for some $0\leq d'\leq\lfloor\frac sr\rfloor$.
\end{thm}

We will show that the assumption \eqref{assumption} is satisfied
when $n=r+1$, see Corollary~\ref{cor} in the last section.

Notation: For $\aa,\bb\in\N^n$ we define 
$\aa\setminus\bb\in\N^n$ by $(\aa\vee\bb)-\bb$, in other words,
$(\aa\setminus\bb)_i:=\max\{a_i-b_i,0\}$.
The support of $\aa$ is defined by
$\supp(\aa):=\{j:a_j>0\}$.

\section{Proof of Theorem~\ref{thm1} --- the case when $n$ is large}
Let $r,s$ be given, and let $s=dr+p$, $0\leq p<r$.
We consider the situation $n\to\infty$ for fixed $r,s,d$, and $p$.

\begin{claim}
$|\DD(\UU(\tilde\ee_p,d))|=
\sum_{j=0}^p\binom pj\binom{n-j+d}d=(2^p/d!)n^d+O(n^{d-1})$.
\end{claim}
\begin{proof}
By definition we have
\[
 \DD(\UU(\tilde\ee_p,d))=\{\xx+\yy\in\N^n:|\xx|\leq d, \,\yy\prec{\tilde\ee_p}\}.
\]
We rewrite the RHS by classifying vectors according to their supports.
For $I\subset[p]$ let ${\tilde\ee_p}|_I$ be the restriction of 
${\tilde\ee_p}$ to $I$, that is,
$({\tilde\ee_p}|_I)_i$ is $1$ if $i\in I$ and $0$ otherwise, and let
\[
R(I):=\{{\tilde\ee_p}|_I+\zz:\supp(\zz)\subset I\sqcup([n]\setminus[p]),\,|\zz|\leq d\}.
\]
Then we have 
$\DD(\UU(\tilde\ee_p,d))=\bigsqcup_{I\subset[p]} R(I)$.
For each $I\in\binom{[p]}i$ 
the number of $\zz$ in $R(I)$ equals the number of nonnegative
integer solutions of $z_1+z_2+\cdots+z_{i+(n-p)}\leq d$.
Thus it follows that $|R(I)|=\binom{n-(p-i)+d}d$, and
\[
|\DD(\UU(\tilde\ee_p,d))|=\sum_{i=0}^p\binom pi\binom{n-(p-i)+d}d
=\sum_{j=0}^p\binom pj\binom{n-j+d}d.
\]
The RHS is further rewritten using
$\binom{n-j+d}d=n^d/d!+O(n^{d-1})$ and
$\sum_{j=0}^p\binom pj=2^p$, as needed.
\end{proof}

Let $A\subset\N^n$ be $r$-wise $s$-union with maximal size. 
So $A$ is a down set. We will show that 
$|A|\leq |\DD(\UU(\tilde\ee_p,d))|$.

First suppose that  {there is a $t$} with $2\leq t\leq r$ such that
$A$ is $t$-wise $(dt+p)$-union, but not $(t-1)$-wise $(d(t-1)+p)$-union.
In this case, by the latter condition, 
there are $\bb_1,\ldots,\bb_{t-1}\in A$ such that
$|\bb|\geq d(t-1)+p+1$, where $\bb=\bb_1\vee\cdots\vee\bb_{t-1}$.
Then, by the former condition, for every $\aa\in A$ 
it follows that $|\aa\vee\bb|\leq dt+p$, so $|\aa\setminus\bb|\leq d-1$.
This gives us 
\[
 A\subset \{\xx+\yy\in\N^n: |\xx|\leq d-1,\, \yy\prec\bb\}.
\]
There are $\binom{n+(d-1)}{d-1}$ 
choices for $\xx$ satisfying $|\xx|\leq d-1$.
On the other hand, the number of $\yy$ with $\yy\prec\bb$ is independent
of $n$ (so it is a constant depending on $r$ and $s$ only). 
In fact $|\bb|\leq (t-1)s<rs$, and there are less than $2^{rs}$ choices 
for $\yy$. Thus we get $|A|<\binom{n+(d-1)}{d-1}2^{rs}=O(n^{d-1})$ 
and we are done.

Next we suppose that 
\begin{equation}\label{dt+p union}
\text{$A$ is $t$-wise $(dt+p)$-union for all $1\leq t\leq r$. }
\end{equation}
The case $t=1$ gives us $|\aa|\leq d+p$ for every $\aa\in A$.
If $p=0$, then this means that $A\subset\DD(\UU(\zero,d))$, 
which finishes the proof for this case. 
So, from now on, we assume that $1\leq p<r$.
We will see that 
there is a $u$ with $u\geq 1$ such that there exist
$\bb_1,\ldots,\bb_u\in A$ satisfying 
\begin{equation}\label{def:u}
 |\bb|=u(d+1),
\end{equation}
where $\bb:=\bb_1\vee\cdots\vee\bb_u$. In fact we have \eqref{def:u} for
$u=1$, if otherwise $A\subset\DD(\UU(\zero,d))$.
On the other hand, 
setting $t=p+1 \leq r$ in \eqref{dt+p union}, 
we see that $A$ is $(p+1)$-wise $((p+1)(d+1)-1)$-union,
and \eqref{def:u} fails if $u=p+1$.
So we choose maximal $u$ with $1\leq u\leq p$ satisfying \eqref{def:u}, and fix 
$\bb=\bb_1\vee\cdots\vee\bb_u$. By this maximality, for every $\aa\in A$,
it follows that $|\aa\vee\bb|\leq(u+1)(d+1)-1$, and
\begin{equation}\label{a-b<=d}
 |\aa\setminus\bb|=|\aa\vee\bb|-|\bb|\leq d.
\end{equation}
Using \eqref{a-b<=d} we have $A\subset\bigcup_{i=0}^d A_i$, where
\[
 A_i:=\{\xx+\yy\in A: |\xx|=i,\, \yy\prec\bb\}.
\]
Then we have $|A_i|\leq \binom{n+i}i2^{|\bb|}$. Noting that 
$|\bb|\leq u(d+1)<r(d+1)=O(1)$ 
it follows $\sum_{i=0}^{d-1}|A_i|=O(n^{d-1})$.
So the size of $A_d$ is essential. 

We naturally identify $\aa\in A$ with a subset of 
$[n]\times\{1,\ldots,d+p\}$. Formally let 
\[
 \phi(\aa):=\{(i,j):1\leq i\leq n,\, 1\leq j\leq a_i\},
\]
for example, if $\aa=(1,0,2)$, then $\phi(\aa)=\{(1,1),(3,1),(3,2)\}$.
Define $m=m(d)$ to be $r+1$ if $d=1$ and $dr$ if $d\geq 2$.
We say that $\bb'\prec\bb$ is \emph{rich} if there exist $m$ vectors
$\cc_1,\ldots,\cc_m$ of weight $d$ 
such that $\bb'\vee\cc_j\in A$ for every $j$, and
the $m+1$ subsets $\phi(\cc_1),\ldots,\phi(\cc_m),\phi(\bb)$ are 
pairwise disjoint. 
In this case $\bb''\vee\cc_j\in A$ for all $\bb''\prec\bb'$
because $A$ is a down set. This means that richness is hereditary, namely,
if $\bb'$ is rich and $\bb''\prec\bb'$, then $\bb''$ is rich as well. 
Informally, $\bb'$ is rich if it can be extended
to a $(|\bb'|+d)$-element subset of $A$ in $m$ ways disjointly outside $\bb$.
We are comparing our family $A$ with the reference family 
$\DD(\UU(\tilde\ee_p),d)$, and we define $\tilde\bb$ which plays the  
role of $\tilde\ee_p$ in our family, namely, let us define
\[
 \tilde\bb:=\bigvee \{\bb'\prec\bb:\text{$\bb'$ is rich}\}.
\]

\begin{claim}\label{q<=p}
$|\tilde\bb|\leq p$.
\end{claim}
\begin{proof}
Suppose the contrary. Then $|\tilde\bb|>p$ and we can find rich 
$\bb'_1,\bb'_2,\ldots,\bb'_{p+1}$ {(with repetition if necessary)} 
such that 
$|\bb'_1\vee\cdots\vee\bb'_{p+1}|\ge p+1$. 
Since richness is hereditary we may assume that
$|\bb'_1\vee\cdots\vee\bb'_{p+1}|= p+1$. 
Let $\cc_1^{(i)},\ldots,\cc_m^{(i)}$ support the richness of $\bb'_i$. 
By definition $\phi(\cc_1^{(i)}),\ldots,\phi(\cc_m^{(i)})$ and $\phi(\bb)$ 
are pairwise disjoint.
Let $\aa_1:=\bb'_1\vee\cc_{j_1}^{(1)}\in A$, say, $j_1=1$.
Then choose $\aa_2:=\bb'_2\vee\cc_{j_2}^{(2)}$ so that
$\phi({\cc_{j_1}^{(1)}})$ and $\phi(\cc_{j_2}^{(2)})$ 
are disjoint. 
If $i\leq p$, then having $\aa_1,\ldots,\aa_i$ chosen, 
we only used $id$ elements as $\bigcup_{l=1}^i\phi(\cc_{j_l}^{(l)})$, 
which intersect at most $id$ of $\cc_1^{(i+1)},\ldots,\cc_m^{(i+1)}$. 
Then, since $id\leq pd< rd\leq m$, 
we still have some $\cc_{j_{i+1}}^{(i+1)}$, which is 
disjoint from any already chosen vectors.
So we can continue this procedure until we get 
$\aa_{p+1}:=\bb'_{p+1}\vee\cc_{j_{p+1}}^{(p+1)}\in A$
such that all 
$\phi(\cc_{j_1}^{(1)}),\ldots,\phi(\cc_{j_{p+1}}^{(p+1)})$ 
and $\phi(\bb)$ are disjoint.   
However, these vectors yield that
\begin{align*}
|\aa_1\vee\cdots\vee\aa_{p+1}|&
=|\bb'_1\vee\cdots\vee\bb'_{p+1}|
+|\cc_{j_1}^{(1)}|+\cdots+|\cc_{j_{p+1}}^{(p+1)}|\\
&= (p+1) + (p+1)d =(p+1)(d+1), 
\end{align*}
which contradicts \eqref{dt+p union} at $t=p+1$.
\end{proof}

If $\yy\prec\bb$ is not rich, then 
\[
\{\phi(\xx):\xx+\yy\in A_d, |\xx|=d\}
\]
is a family of $d$-element subsets on $(d+p)n$ vertices, which has no
$m$ pairwise disjoint subsets (so the matching number is $m-1$ or less).
Thus, by the Erd\H os matching theorem \cite{Erdos}, the size of this family is 
$O(n^{d-1})$. 
There are at most $2^{|\bb|}=O(1)$ choices for non-rich $\yy\prec\bb$, and 
we can conclude that the number of vectors
in $A_d$ coming from non-rich $\yy$ is $O(n^{d-1})$.
Then the remaining vectors in $A_d$ come from rich $\yy\prec\tilde\bb$, and
the number of such vectors is at most $2^{|\tilde\bb|}\binom{n+d}d$.
Note also that $\sum_{i=0}^{d-1}|A_i|=O(n^{d-1})$.
Consequently we get
\[
 |A|\leq 2^{|\tilde\bb|}\binom{n+d}d+O(n^{d-1})
=(2^{|\tilde\bb|}/d!)\,n^d+O(n^{d-1}).
\]
Recall that the reference family is of size 
$(2^p/d!)n^d+O(n^{d-1})$, and 
$|\tilde\bb|\leq p$ from Claim~\ref{q<=p}.
So we only need to deal with the case when $|\tilde\bb|=p$ and
there are exactly $2^p$ rich sets.
In other words, $\tilde\bb=\tilde\ee_p$ (by renaming coordinates if necessary)
and every $\bb'\prec\tilde\ee_p$ is rich.
We show that $A\subset\DD(\UU(\tilde\ee_p,d))$. Suppose the contrary, then
there is an $\aa\in A$ such that 
$|\aa'|\geq d+1$, where $\aa'=\aa\setminus\tilde\ee_p$.
Since $A$ is a down set we may assume that $|\aa'|=d+1$.
Now $\tilde\ee_p$ is rich and
let $\cc_1,\ldots,\cc_m$ be vectors assured by the richness.
We remark that $m-(d+1)\geq r-1$. In fact if $d=1$ then
$m-(d+1)
=r-1$, and if $d\geq 2$ then 
$m-(d+1)
=(r-1)(d-1)+r-2\geq r-1$.
So we may assume that $\phi(\cc_1),\ldots,\phi(\cc_{r-1})$ 
are pairwise disjoint and disjoint to $\phi(\aa)$ as well. 
Let $\aa_i:=\tilde\ee_p\vee\cc_i\in A$ for $1\leq i\leq r-1$.
Then we get 
\begin{align*}
|\aa\vee\aa_1\vee\cdots\vee\aa_{r-1}|
&=|\tilde\ee_p\vee\aa'|+|\cc_1|+\cdots+|\cc_{r-1}|\\
&=(p+d+1)+(r-1)d=dr+p+1=s+1, 
\end{align*}
which contradicts that $A$ is $r$-wise $s$-union. 
This completes the proof of Theorem~\ref{thm1}.

\section{The polytope $\bP$ and proof of Theorem~\ref{mainthm}}
Let $\aa=(a_1,\ldots,a_n)\in\N^n$ with $|\aa|=s-rd$ for some $d\in\N$.
We introduce a convex polytope $\bP\subset\R^{n}$, 
which will play a key role in our proof.
This polytope is defined by the following 
$n+\binom n1+\binom n2+\cdots +\binom {n}{n-r+1}$ inequalities:
\begin{align}
 x_i&\geq 0 &&\hskip -5em\text{if } 1\leq i\leq n,\label{L:eq1}\\
 \sum_{i\in I}x_i&\leq \sum_{i\in I}a_i+d 
  &&\hskip -5em\text{if }1\leq |I|\leq n-r+1,\, I\subset [n].\label{L:eq2}
\end{align}
Namely, 
\[
 \bP:=\{\xx\in\R^{n}: \text{$\xx$ satisfies \eqref{L:eq1} and \eqref{L:eq2}}\}.
\]
Let $L$ denote the integer lattice points in $\bP$:
\[
L= L(r,n,\aa,d):=\{\xx\in\N^{n}:\xx\in\bP\}. 
\]

\begin{la}\label{lemma1}
The two sets $K$ (defined by \eqref{def:K})
and $L$ are the same, and $r$-wise $s$-union.
\end{la}
\begin{proof}
This lemma is a consequence of the following three claims.

\begin{claim}\label{claim1}
The set $K$ is $r$-wise $s$-union.
\end{claim}

\begin{proof}
 Let $\xx_1,\xx_2,\ldots,\xx_r\in K$. We show that 
$|\xx_1\myor\xx_2\myor\cdots\myor\xx_r|\leq s$.
We may assume that $\xx_j\in\UU(\aa+i_j\one,d-ui_j)$, where $u=n-r+1$.
We may also assume that $i_1\geq i_2\geq \cdots \geq i_r$.
Let $\bb:=\aa+i_1\one$. Then, informally, 
$|\xx\setminus\bb|:=|(\xx\myor\bb)-\bb|$
counts the excess of $\xx$ above $\bb$, more precisely, it is
$\sum_{j\in[n]}\max\{0,x_j-b_j\}$. Thus we have
\begin{align*}
|\xx_1\myor\xx_2\myor\cdots\myor\xx_r| &\leq
|\bb|+\sum_{j=1}^r|\xx_j\setminus\bb|\\
&\leq|\aa|+ni_1+\sum_{j=1}^r\big((d-ui_j)-(i_1-i_j)\big)\\
&=|\aa|+dr+(n-r)i_1-\sum_{j=1}^r(u-1)i_j\\
&=s-(n-r)\sum_{j=2}^ri_j\leq s,
\end{align*}
as required.
\end{proof}

\begin{claim}\label{claim2}
$K\subset L$.
\end{claim}
\begin{proof}
Let $\xx\in K$. We show that $\xx\in L$, that is,
$\xx$ satisfies \eqref{L:eq1} and \eqref{L:eq2}.
Since \eqref{L:eq1} is clear by definition of $K$, 
we show that \eqref{L:eq2}. To this end we may assume that
$\xx\in\UU(\aa+i\one,d-ui)$, where $u=n-r+1$ and $i\leq\lfloor\frac du\rfloor$.
Let $I\subset[n]$ with $1\leq|I|\leq u$. Then $i|I|\leq ui$. Thus it follows
\[
 \sum_{j\in I}x_j\leq\sum_{j\in I}a_j+i|I|+(d-ui)\leq
\sum_{j\in I}a_j+d,
\]
which confirms \eqref{L:eq2}.
\end{proof}

\begin{claim}\label{claim3}
$K\supset L$.
\end{claim}
\begin{proof}
Let $\xx\in L$. We show that $\xx\in K$, that is,
there exists some $i'$ such that
$0\leq i'\leq \lfloor\frac d{n-r+1}\rfloor$ and
\[
 |\xx \setminus (\aa+i'\one)| \leq d-(n-r+1)i'.
\]
We write $\xx$ as
\[
 \xx=(a_1+i_1,a_2+i_2,\ldots,a_{n}+i_{n}),
\]
where we may assume that
$d\geq i_1\geq i_2\geq \cdots \geq i_{n}$.
We notice that some $i_j$ can be negative.
Since $\xx\in L$ it follows from \eqref{L:eq2} 
(a part of the definition of $L$) that 
if $1\leq|I|\leq n-r+1$ and $I\subset[n]$, then
\[
 \sum_{j\in I}i_j\leq d.
\]
Let $J:=\{j:x_j\geq a_j\}$ and we argue separately by the size of $|J|$.

If $|J|\leq n-r+1$, then we may choose $i'=0$. In fact,
\begin{align*}
 |\xx\setminus\aa|&=\max\{0,i_1\}+\max\{0,i_2\}+\cdots+\max\{0,i_{n-r+1}\}\\
 &=\max\bigg\{\sum_{j\in I}i_j: I\subset [n-r+1]\bigg\}\leq d.
\end{align*}

If $|J|\geq n-r+2$, then we may choose $i'=i_{n-r+2}$.
In fact, by letting $i':=i_{n-r+2}$, we have
\begin{align*}
|\xx\setminus(\aa+i'\one)|&=
(i_1-i')+(i_2-i')+\cdots+(i_{n-r+1}-i')\\
&\leq d-(n-r+1)i'.
\end{align*}
We need to check $0\leq i'\leq\lfloor\frac d{n-r+1}\rfloor$.
It follows from $|J|\geq n-r+2$ that $i'\geq 0$.
Also $d\geq i_1\geq i_2\geq \cdots\geq i_{n-r+2}$ and
$i_1+i_2+\cdots+i_{n-r+1}\leq d$ yield
$i'\leq\lfloor\frac d{n-r+1}\rfloor$.
\end{proof}
This completes the proof of Lemma~\ref{lemma1}.
\end{proof}

Let 
\[
 \sigma_k(\aa):=\sum_{K\in\binom{[n]}k}\prod_{i\in K}a_i
\]
be the $k$th elementary symmetric polynomial of $a_1,\ldots,a_n$.

\begin{la}\label{lemma2}
The size of $K(r,n,\aa,d)$ is given by 
\begin{align*}
|K(r,n,\aa,d)|&=
\sum_{j=0}^{n}\binom{d+j}j\sigma_{n-j}(\aa)\\
&\qquad+\sum_{i=1}^{\lfloor\frac du\rfloor}
\sum_{j=u+1}^{n} 
\left(
\binom{d-ui+j}j-\binom{d-ui+u}j
\right)\sigma_{n-j}(\aa+i\one),
\end{align*}
where $u=n-r+1$.
Moreover, for fixed $n,r,d$ and $|\aa|$, this size is maximized
if and only if $\aa$ is a balanced partition.
\end{la}

\begin{proof}
For $J\subset[n]$ let $\xx|_J$ be the restriction of $\xx$ to $J$, that is,
$(\xx|_J)_i$ is $x_i$ if $i\in J$ and $0$ otherwise.

First we count the vectors in the base layer $\DD(\UU(\aa,d))$.
To this end we partition this set into $\bigsqcup_{J\subset[n]}A_0(J)$, where 
\[
 A_0(J)=\{\aa|_J+\ee+\bb:\supp(\ee)\subset J,\,|\ee|\leq d,\, 
\supp(\bb)\subset[n]\setminus J,\, b_i<a_i\text{ for }i\not\in J\}.
\]
The number of vectors $\ee$ with the above property is equal to the number of
non-negative integer solutions of the inequality
$x_1+x_2+\cdots+x_{|J|}\leq d$, which is $\binom{d+|J|}{|J|}$.
The number of vectors $\bb$ is clearly $\prod_{l\in[n]\setminus J}a_l$.
Thus we get
\[
\sum_{J\in\binom{[n]}j} |A_0(J)|=
\sum_{J\in\binom{[n]}j}\binom{d+|J|}{|J|}\prod_{l\in[n]\setminus J}a_l
=\binom{d+j}{j}\sigma_{n-j}(\aa),
\]
and $|\DD(\UU(\aa,d))|=\sum_{j=0}^n\binom{d+j}{j}\sigma_{n-j}(\aa)$.

Next we count the vectors in the $i$th layer:
\[
 \DD(\UU(\aa+i\one,d-ui))\setminus
 \left(\bigcup_{j=0}^{i-1}\DD(\UU(\aa+j\one,d-uj))\right).
\]
For this we partition the above set into $\bigsqcup_{J\subset[n]}A_i(J)$, where
\begin{align*}
A_i(J)=\{(\aa+i\one)|_J+\ee+\bb:&\supp(\ee)\subset J,\,
d-u(i-1)-|J|<|\ee|\leq d-ui,\\
&\supp(\bb)\subset [n]\setminus J,\, b_l<a_l+i\text{ for }l\not\in J\}.
\end{align*}
In this case we need $d-u(i-1)<|J|+|\ee|$ because the vectors satisfying
the opposite inequality are already counted in the lower layers 
$\bigcup_{j<i}A_j(J)$. 
We also notice that $d-u(i-1)-|J|<d-ui$ implies that $|J|>u$.
So $A_i(J)=\emptyset$ for $|J|\leq u$.
Now we count the number of vectors $\ee$ in $A_i(J)$, or equivalently,
the number of non-negative integer solutions of 
\[
d-u(i-1)-|J|<x_1+x_2+\cdots+x_{|J|}\leq d-ui.
\]
This number is $\binom{d-ui+j}j-\binom{d-ui+u}j$, where $j=|J|$.
On the other hand, the number of vectors $\bb$ in $A_i(J)$ is
$\prod_{l\in[n]\setminus J}(a_l+i)$. Consequently we get
\[
 \sum_{J\subset[n]}|A_i(J)|=
 \sum_{j=u+1}^n\left(\binom{d-ui+j}j-\binom{d-ui+u}j\right)
\sigma_{n-j}(\aa+i\one).
\]
Summing this term over $1\leq i\leq \lfloor\frac du\rfloor$ we finally 
obtain the second term of the RHS of $|K|$ in the statement of this lemma.
Then, for fixed $|\aa|$, the size of $K$ is maximized when $\sigma_{n-j}(\aa)$
and $\sigma_{n-j}(\aa+i\one)$ are maximized. By the property of symmetric
polynomials, this happens if and only if $\aa$ is a balanced partition,
see e.g., Theorem~52 in section 2.22 of \cite{HLP}.
\end{proof}

\begin{proof}[Proof of Theorem~\ref{mainthm}]
Let $A\subset\N^n$ be an $r$-wise $s$-union with \eqref{assumption}.
For $I\subset[n]$ let
\[
 m_I:=\max\bigg\{\sum_{i\in I}x_i:\xx\in A\bigg\}.
\]

\begin{claim}
If $I\subset[n]$ and $1\leq|I|\leq n-r+1$, then
\[
 m_I=\sum_{i\in I}a_i+d.
\]
\end{claim}
\begin{proof}
Choose $j\in I$. By \eqref{assumption} we have $P_j\in A$ and
\begin{equation}\label{eq:m_I}
 m_I\geq\sum_{i\in I}(P_j)_i=\sum_{i\in I}a_i+d.
\end{equation}
We need to show that this inequality is actually an equality.
Let $[n]=I_1\sqcup I_2\sqcup\cdots\sqcup I_r$ be a partition of $[n]$.
Then it follows that
\begin{align*}
s\geq m_{I_1}+m_{I_2}+\cdots+ m_{I_r}
\geq\sum_{i\in[n]}a_i+rd = s,
\end{align*}
where the first inequality follows from the $r$-wise $s$-union property of
$A$, and the second inequality follows from \eqref{eq:m_I}. Since the left-most
and the right-most sides are the same $s$, we see that all inequalities are 
equalities. This means that \eqref{eq:m_I} is equality, as needed.
\end{proof}
By this claim if $\xx\in A$ and $1\leq|I|\leq n-r+1$, then we have
\[
 \sum_{i\in I}x_i\leq m_I=\sum_{i\in I}a_i+d.
\]
This means that $A\subset L$.
Finally the theorem follows from Lemmas~\ref{lemma1} and \ref{lemma2}.
\end{proof}

\begin{cor}\label{cor}
If $n=r+1$, then Conjecture is true.
\end{cor}
 
\begin{proof}
Let $n=r+1$ and let $A\subset\N^{r+1}$ be $r$-wise $s$-union with maximum size.
Define $\mm$ by \eqref{def:mi}.
Since $n-r=1$ we can define $d$ by \eqref{def:d}.
Then define $\aa$ by \eqref{def:ai}. 
We need to verify $a_i\geq 0$ for all $i$.
To this end we may assume that $m_1\geq m_2\geq\cdots\geq m_{r+1}$. 
Then $a_i\geq a_{r+1}=m_{r+1}-d$, so it suffices to show $m_{r+1}\geq d$.
Since $A$ is $r$-wise $s$-union it follows that $m_1+m_2+\cdots+m_r\leq s$.
This together with the definition of $d$ implies
$d=|\mm|-s\leq m_{r+1}$, as needed.
So we can properly define $P_i$ by \eqref{def:Pi}.

Next we check that $\xx\in A$ satisfies \eqref{L:eq1} and \eqref{L:eq2}.
By definition we have $x_i\leq m_i=a_i+d$, so we have \eqref{L:eq1}.
Since $A$ is $r$-wise $s$-union, we have 
\[
 (x_1+x_2)+m_3+\cdots+m_{r+1}\leq s,
\]
or equivalently, 
\[
 (x_1+x_2)+(a_3+d)+\cdots+(a_{r+1}+d)\leq s=|\aa|+rd.
\]
Rearranging we get $x_1+x_2\leq a_1+a_2+d$, and we get the other cases 
similarly, so we obtain \eqref{L:eq2}. Thus $A\subset L$ 
and the result follows from Lemmas~\ref{lemma1} and \ref{lemma2}.
\end{proof}

\section*{Acknowledgement}
The authors thank the anonymous referees of European journal of combinatorics 
for pointing out the error in Claim~1 of the earlier version and for many helpful suggestions
which improve the presentation of this paper.




\end{document}